\documentclass[11pt]{article}
\usepackage[utf8]{inputenc} 
\usepackage[left=1in,right=1in,top=1.15in,bottom=1.25in]{geometry}
\usepackage{authblk}
\usepackage[small]{titlesec}
\usepackage{amsmath,amssymb,graphicx,booktabs}
\usepackage{times} 
\usepackage{multirow}
\usepackage[dvipsnames,table,dvipsnames*, svgnames*, hyperref]{xcolor}
\usepackage[colorlinks=true,linkcolor=blue,citecolor=blue,urlcolor=blue,breaklinks]{hyperref}

\usepackage{amssymb,bm,bbm} 	
\DeclareUnicodeCharacter{202F}{\,}

\newtheorem{theorem}{Theorem}[section]

\newtheorem{lemma}[theorem]{Lemma}

\newtheorem{proposition}[theorem]{Proposition}
\newtheorem{remark}[theorem]{Remark}

\newtheorem{assumption}{Assumption}[section]
\newcommand{\zb}{\mathbf{z}}
\newcommand{\sfc}{\mathsf{c}}

\newcommand{\dd}{\mathrm{d}}
\newcommand{\OO}{\mathrm{O}}
\newcommand{\oo}{\mathrm{o}}
\newcommand{\wls}{\mathrm{wls}}
\newcommand{\SNR}{\mathsf{SNR}}
\newcommand{\sfx}{\mathsf{x}}
\newcommand{\calS}{\mathcal{S}}
\newenvironment{proof}[1][Proof]{\noindent\textbf{#1.} }{\ \rule{0.5em}{0.5em}}
\numberwithin{equation}{section}
\numberwithin{figure}{section}

\date{}

\begin{document}
\title{\bf Phase transitions and approximations of mean squared error for state-space models with fractional differencing} 


\author[1]{Prabir Burman \thanks{Email: pburman@ucdavis.edu.}}
\author[1]{Xiucai Ding \thanks{Email: xcading@ucdavis.edu.  The author is partially supported by NSF DMS-2515104.}}
\author[1]{
Robert H. Shumway \thanks{Email: rhshumway@ucdavis.edu.}}

\affil[1]{Department of Statistics, University of California, Davis}

\maketitle
\vspace{-30pt}
\begin{center}
      \textit{This paper is dedicated to the memory of our colleague, Professor Robert H. Shumway (1936–2025), whose vision, mentorship, and lifelong contributions continue to inspire us.}
\end{center}

\vspace{5pt}
\begin{abstract}
We study trend estimation in state-space models in which the trend has a fractional stochastic difference of order $d>0$ and the observation errors form a short-range-dependent stationary process. Using finite-sequence fractional summation and differencing operators, we analyze the penalized least-squares estimator obtained by shrinking the fractional differences of the trend. We derive asymptotic mean squared error (MSE) approximations for all $d>0$ and identify a sharp phase transition at $d=1/2$. When $d>1/2$, the estimator is consistent and its optimally balanced MSE has order $n^{-(2d-1)/(2d)}$. At the boundary $d=1/2$, we obtain a refined finite-sample approximation and show that the MSE decreases at the slower order $\log\log n/\log n$. When $0<d<1/2$, the MSE converges to an explicit positive limit, so consistent recovery of the trend is impossible under the considered scaling. We also describe a practical criterion for choosing the penalty parameter and differencing order, and numerical experiments illustrate the MSE approximations and the behavior of the selection procedure.
\end{abstract}

\section{Introduction}
Estimating the trend of a time series is a fundamental task in many applications like economic and financial forecasting \cite{hatanaka1996,tsay2010analysis}, climate modeling \cite{MUDELSEE2019310}, biomedical monitoring \cite{patharkar2024predictive}, and public health surveillance \cite{ensor2024online}. In the time series analysis literature \cite{brockwell2016introduction,commandeur2007introduction,durbin2000,shumway2025}, the following state-space model for the observations $Y_t$, $1 \leq t \leq n$, has received significant attention due to its simplicity and practical interpretability
\begin{equation}\label{eq_basicmodel}
Y_{t}=\mu_{t}+\varepsilon_{t}, \ t=1,..,n, 
\end{equation}
where $\{\mu_t\}$ denotes the trend component, and $\{\varepsilon_t\}$ is assumed to be a zero-mean stationary process with variance $\sigma^2_\varepsilon$ which is independent of $\{\mu_t\}$.

Modeling the trend component in equation (\ref{eq_basicmodel}) can be broadly classified into two categories: deterministic $\mu_t$ and stochastic $\mu_t$. Deterministic trends can be modeled either parametrically, when appropriate, or nonparametrically using techniques such as kernel smoothing, local polynomial regression, and sieve expansions. For a detailed review from this perspective, we refer readers to the monograph \cite{fan2003}. We point out that, in the nonparametric modeling, it is typically assumed that the trend $\mu_t$ can be expressed as $\mu(t/n)$, where $\mu(\cdot)$ is a smooth function defined on $[0,1]$. It is well known that if $\mu$ is $q$ times differentiable and its $p$th derivative satisfies a Lipschitz condition of order $\theta$, then the optimal mean squared error (MSE) of a nonparametric estimator of $\mu(\cdot)$ is of order $\mathrm{O}(n^{-2p/(2p+1)})$, where $p = q + \theta$ \cite{stone1982}.

Stochastic trend models are widely used by engineers, statisticians, and economists \cite{box2015time, durbin2000, hatanaka1996, shumway2025} due to their modeling flexibility and ability to capture long-memory behavior in time series data. In particular, $\mu_t$ is modeled using fractional differencing, meaning that the $d$th difference of the trend consists of mean-zero i.i.d. (or stationary) random variables, where $d > 0$ is a real number representing the order of differencing \cite{burman_shumway2009estimation,hartl2025robust,hassler2018time,nielsen2015asymptotics, palma2007long}. To be more specific, one assumes that for some $d>0$
\begin{equation}\label{eq_differencestandardform}
\nabla^d \mu_t=\gamma_t, 
\end{equation}
where $\nabla$ is the difference operator and $\{\gamma_t\}$ are some mean zero i.i.d. random variables with variance $\sigma^2_{\gamma}.$ The model (\ref{eq_differencestandardform}) and its variants have been widely used in the literature to model stochastic trends. For example, see \cite{burman_shumway2009estimation, CHANAOS, hang2006estimation, durbin2000, grassi2014long,haldrup2007estimation, hartl2022approximate, hassler2018time}.  

We point out that the system defined by (\ref{eq_basicmodel}) and (\ref{eq_differencestandardform}) constitutes a state space model. In the literature, a common approach to obtaining a solution $\bm{\mu} = (\mu_1, \cdots, \mu_n)^\top$ is to minimize the following objective:
\begin{equation}\label{eq_optimizationproblem}
\sum_{t=1}^n (Y_t - \mu_t)^2 + \nu \sum_{t=r}^n (\nabla^d \mu_t)^2, \quad \nu := \frac{\sigma_\varepsilon^2}{\sigma_\gamma^2},
\end{equation}
where $r \equiv r(d)$ is an integer depending on $d$. In particular, when $d$ is an integer, one typically chooses $r = d+1$, as in \cite{burman_shumway2009estimation}. A key technical advantage of minimizing (\ref{eq_optimizationproblem})—as will be demonstrated in Section \ref{sec_setupestimation}—is that it admits a closed-form solution, which is closely related to the Kalman filter \cite{burman_shumway2009estimation, hartl2025robust}. This connection allows us to analyze its asymptotic properties.

In this paper, we aim to investigate the theoretical performance of the solution obtained by minimizing (\ref{eq_optimizationproblem}) for the model (\ref{eq_basicmodel}) under the setup in (\ref{eq_differencestandardform}) for all $d > 0$. Specifically, we study how the asymptotic behavior and phase transitions of the mean squared error (MSE) of the estimates depend on the parameters $d$ and $\nu$, where $\nu$ is closely related to the signal-to-noise ratio; see Section \ref{sec_setupestimation} for a more precise formulation.  We now provide an overview of our main results. We show that, under mild conditions on $\nu$, the asymptotic behavior of the mean squared error (MSE) exhibits a phase transition that can be fully characterized by $d$ for all $d > 0$. First, as shown in part (1) of Theorem \ref{thm_mainthm}, when $d > 1/2$, we obtain a consistent estimator whose MSE can be of order $\mathrm{O}\left( n^{-(2d - 1)/2d} \right)$. Second, as stated in part (2) of Theorem \ref{thm_mainthm}, when $d = 1/2$, the estimator from (\ref{eq_optimizationproblem}) remains consistent, but with a much slower convergence rate of order $\mathrm{O}\left( \log \log n / \log n \right)$. Finally, when $0 < d < 1/2$, (\ref{eq_optimizationproblem}) fails to yield a consistent estimator under the scaling considered here. In this case, we explicitly compute the nontrivial limiting value of the MSE, as presented in part (3) of Theorem \ref{thm_mainthm}. In addition, when consistent estimators can be obtained by minimizing (\ref{eq_optimizationproblem}) (i.e., $d \geq 1/2$), we also propose a data-driven approach to select the practically unknown parameters $d$ and $\nu$, as detailed in Section \ref{sec_selectionofparameters}. This method demonstrates superior performance in our numerical studies, as shown in Section \ref{sec_numericalstudy}.

We now discuss several related works in the literature. In \cite{burman_shumway2009estimation}, the authors studied the problem when $d \geq 1$ is an integer. They showed that for $d \in \mathbb{N}^+$ and divergent $\nu$, the optimization problem in (\ref{eq_optimizationproblem}) yields a consistent estimator, and the mean squared error (MSE) can converge at a rate of $\mathrm{O}\left( n^{-(2d - 1)/2d} \right)$. However, several challenges arise when extending their results to general $d > 0$. For example, one must appropriately choose the parameter $r$ in (\ref{eq_optimizationproblem}) and define fractional differences for general $d > 0$ on finite sequences. Additionally, when $d < \frac{1}{2}$, the estimator becomes inconsistent, and the analysis in \cite{burman_shumway2009estimation} does not directly apply. Computing the non-vanishing asymptotic MSE in this case requires new tools. On the other hand, in \cite{CHANAOS, hang2006estimation, hartl2025robust, hartl2022approximate, nielsen2015asymptotics}, the authors consider a closely related fractionally integrated process instead of (\ref{eq_differencestandardform}), and establish consistency of the parameter estimates under various assumptions on ${\varepsilon_t}$. However, the asymptotic behavior of the mean squared error (MSE) and its phase transitions have not been investigated within their framework.   

Before concluding this section, we offer two remarks on possible future generalizations of our results and techniques. First, in the current paper, we study the performance of the estimators by minimizing (\ref{eq_optimizationproblem}), which is closely related to the likelihood function of (\ref{eq_basicmodel}) under the i.i.d. Gaussian assumption for $\{\varepsilon_t\}$ \cite{hang2006estimation}. An alternative to (\ref{eq_optimizationproblem}), especially when $\{\varepsilon_t\}$ is not i.i.d., is to consider a weighted least squares estimator that incorporates the covariance structure of $\{\varepsilon_t\}$. Analogous results on the phase transitions of the asymptotic MSE have been established in \cite{burman_shumway2009estimation} to some extent for $d \in \mathbb{N}^+$. We believe our techniques can be extended to general $d > 0$; see Remark \ref{rem_generalizationweighted} for further discussion. Second, although we focus on the stationary setup, we believe our results can be extended to cases where $\{\varepsilon_t\}$ in (\ref{eq_basicmodel}) or $\{\gamma_t\}$ in (\ref{eq_differencestandardform}) is locally stationary, as considered in \cite{ding2023autoregressive,ding_zhou2025pacf}. We plan to pursue this direction in future work.

The remainder of the paper is organized as follows. Section \ref{sec_setupestimation} introduces the basic setup and presents the explicit solution to the optimization problem (\ref{eq_optimizationproblem}). Section \ref{sec_mainresultstechinical} states our main results. In Section \ref{sec_selectionofparameters}, we discuss strategies for selecting the practically unknown parameters $\nu$ and $d$ when consistent estimation is feasible. Numerical simulations are presented in Section \ref{sec_numericalstudy}, and technical proofs are provided in Appendices \ref{appendxi_mainprop} and \ref{appendix_technical_lemmas}.

\vspace{3pt}

\noindent {\bf Conventions.} For two sequences of deterministic positive values $\{a_n\}$ and $\{b_n\}$, we write $a_n=\mathrm{O}(b_n)$ if $a_n \leq C b_n$ for some positive constant $C>0$. We write $a_n \asymp  b_n$ if $a_n=\OO(b_n)$ and $b_n=\OO(a_n).$ Moreover, we write $a_n=\mathrm{o}(b_n)$ if $a_n \leq c_n b_n$ for some positive sequence $c_n \downarrow 0.$ We use $\mathbb{N}$ to denote the set of natural numbers (including zero), and $\mathbb{N}^+$ to denote the set of positive natural numbers.

\section{Some background, the assumptions and basic setup}\label{sec_setupestimation}

In this section, we provide the basic setup, assumptions and some background on (\ref{eq_differencestandardform}) and the closed form solution of (\ref{eq_optimizationproblem}). We start with introducing an important operator from \cite{BURMAN2006677}. For $d>0$ and any sequence $\zb=(z_{1},...,z_{n})^\top \in \mathbb{R}^n,$ we denote the operator $\calS_d: \mathbb{R}^n \rightarrow \mathbb{R}^n$ that for all $1 \leq t \leq n$  
\begin{equation}\label{eq_definitionoperator}
(\calS_{d}\zb)(t)=\sum_{j=1}^t (-1)^{t-j}\binom{-d}{t-j}z_{j},
\end{equation} 
where $\binom{-d}{t-j}$ is the generalized binomial coefficient. The operator defined in (\ref{eq_definitionoperator}) possesses several important properties, as established in \cite[Lemma 1]{BURMAN2006677}, for example, for any real numbers $\alpha, \beta \in \mathbb{R},$
\begin{equation}\label{eq_operatorproperty1}
\calS_0=\mathcal{I}, \ \ \ \calS_{\alpha} \calS_{\beta}=\calS_{\alpha+\beta},
\end{equation}  
which implies that for $d>0$
\begin{equation}\label{eq_operatorproperty2}
\calS_{-d}=\calS_d^{-1}. 
\end{equation}
We point out that when $d>0, \ $ $\calS_{-d}$  coincides with the  difference operator \cite{gray1988new}. Moreover, in matrix form, $\calS_{-d} \zb$ can be rewritten as
\begin{equation*}
\calS_{-d} \zb = S \zb,
\end{equation*}
where $S \in \mathbb{R}^{n \times n}$ is a lower triangular matrix with entries defined by $S_{tj} = (-1)^{t-j} \binom{-d}{t-j}$ for $1 \leq t \leq n$ and $j \leq t$. 

Denote $S^{(r-1)}$ be the $(n+1-r) \times n$ matrix obtained by deleting the first $(r-1)$ rows of $S.$ With the above definitions and the conventions $\bm{\mu} = (\mu_1, \cdots, \mu_n)^\top$ and $\bm{Y} = (Y_1, \cdots, Y_n)^\top$, we can rewrite equation (\ref{eq_optimizationproblem}) as follows
\begin{equation}\label{eq_miminizationreducedform}
\|\bm{Y}-\bm{\mu}\|_2^{2}+\nu\sum_{t=r}^n \left[(\calS_{-d}\bm{\mu})(t)\right]^{2}=\|\bm{Y}-\bm{\mu}\|_2^{2}+\nu \bm{\mu}^\top (S^{(r-1)})^\top S^{(r-1)} \bm{\mu}.
\end{equation}
Minimizing the above quantity with respect to $\bm{\mu},$ we obtain the estimate
\begin{equation}\label{eq_estimatefinal}
\widehat{\bm{\mu}}=(I+\nu (S^{(r-1)})^\top S^{(r-1)})^{-1} \bm{Y},
\end{equation}
where $I$ is the $n \times n$ identity matrix. For $d>0$, we set $r=\lceil d \rceil + 1$, where $\lceil d \rceil$ denotes the ceiling of $d$. Therefore, the main objective of this paper is to study $\mathbb{E}\|\widehat{\bm{\mu}} - \bm{\mu}\|_2^2$.

Before introducing our assumptions, we first provide some discussion on the model (\ref{eq_basicmodel}). Following the discussion in \cite{burman_shumway2009estimation}, for $\mu_t, \ 1 \leq t \leq n,$ satisfying (\ref{eq_differencestandardform}), one can solve and express it in an alternative form as follows
\begin{align}\label{eq_mu1part}
\mu_t& =\sum_{j=0}^{r-2} \beta_{r-2-j} (\calS_{d-1-j} \bm{1})(t)+(\calS_d \bm{\gamma})(t), \nonumber \\
&:=\mu_{1t}+\mu_{2t}, 
\end{align}
where $\bm{1} \in \mathbb{R}^n$ is the vector with all entries equal to one, $\bm{\gamma}=(\gamma_1, \cdots, \gamma_n)^\top$ and \(\{\beta_j\}_{j=0}^{r-2}\) are constants determined by the initial conditions of the fractional-difference equation. Based on the above decomposition, we can proceed to discuss the overall signal-to-noise ratio (SNR) for the model (\ref{eq_basicmodel}), denoted as 
\begin{equation}\label{eq_defnsnr}
\SNR:=\frac{\mathbb{E} \|\bm{\mu} \|_2^2}{\mathbb{E} \| \bm{\varepsilon} \|_2^2},
\end{equation}     
where $\bm{\mu}=(\mu_1, \cdots, \mu_n)^\top$ and $\bm{\varepsilon}=(\varepsilon_1, \cdots, \varepsilon_n)^\top.$ Due to stationarity, we have that $\mathbb{E}\|\bm{\varepsilon} \|_2^2=n\sigma_\varepsilon^2.$ Moreover, by a discussion similar to equation (2.1) of \cite{burman_shumway2009estimation}, one can readily obtain that 
\begin{equation}\label{eq_basicdefinition}
\mathbb{E}\| \bm{\mu}_1 \|_2^2 \asymp \sum_{j=0}^{r-2} \beta^2_{r-2-j}n^{2(d-j-0.5)}, \ \ \mathbb{E} \| \calS_d \bm{\gamma} \|_2^2 \asymp \left\{
\begin{array}{cc}
n^{2d} \sigma_{\gamma}^{2} & d>1/2 \\ 
 (n\log n) \sigma_{\gamma}^{2} & d=1/2 \\ 
n \sigma_{\gamma}^{2} & 0<d<1/2%
\end{array}
\right. , 
\end{equation} 
where $\bm{\mu}_1 \in \mathbb{R}^n$ is defined using the $\mu_{1t}, 1 \leq t \leq n,$ in (\ref{eq_mu1part}). 

To avoid repetition, we now summarize the main assumptions of this paper as follows. 
\begin{assumption}\label{assum_mainassumption}
Throughout the paper, we assume the following holds for the model (\ref{eq_basicmodel}). 
\begin{enumerate}
\item We assume that $\{\varepsilon_t\}$ is a zero-mean stationary process with $0<\sigma_\varepsilon^2<\infty$. Furthermore, we assume that for some constant $\delta > 0$, its autocovariance function satisfies $ \rho(|h|) = \OO(|h|^{-2-\delta})$ as $|h| \rightarrow \infty$. 

\item We assume that $\{\mu_t\}$ satisfies (\ref{eq_differencestandardform}) with some $d > 0$, where $\{\gamma_t\}$ are i.i.d. mean-zero random variables, independent of $\{\varepsilon_t\}$, with variance $\sigma_\gamma^2 \equiv \sigma_\gamma^2(d,n)$, and for some constant $\tau > 0$, we have
\begin{equation}\label{eq_sigmaform}
\sigma_{\gamma}^{2}=\frac{\tau^{2}}{(n^{2d-1}-1)/(2d-1)}, \ \ d>0. 
\end{equation}
For $d = 1/2$, the right-hand side of (\ref{eq_sigmaform}) is defined as the limit as $d \to 1/2$. Furthermore, we assume that in (\ref{eq_basicdefinition}),  $\mathbb{E}\| \bm{\mu}_1 \|_2^2 / n = \OO(1)$.
\end{enumerate}

%
%
%
%
\end{assumption}

First, the assumptions imposed on ${\varepsilon_t}$ ensure that the noise component forms a short-range dependent stationary process— a mild and commonly adopted condition in the literature \cite{burman_shumway2009estimation,hang2006estimation,commandeur2007introduction,durbin2000,haldrup2007estimation,hartl2025robust}.

Second, (\ref{eq_sigmaform}) is simply a reformulation of the second term in (\ref{eq_basicdefinition}) using the Box-Cox transform. It ensures that $n^{-1}\mathbb{E}\|\calS_d \bm{\gamma}\|_2^2 \asymp 1$, which, without loss of generality, facilitates a nontrivial discussion of the signal-to-noise ratio (SNR). In fact, we emphasize that the assumptions and scalings imposed here are mild and primarily serve to ensure that the signal-to-noise ratio (SNR) defined in (\ref{eq_defnsnr}) lies in a nontrivial regime, i.e., $0 < \SNR < \infty$. On the one hand, if $\SNR$ approaches zero—indicating that the signal from the trend is completely dominated by noise—then reliable estimation of the trend from $\{Y_t\}$ becomes unlikely. This scenario arises when $\sigma^2_\varepsilon \gg \mathbb{E}\| \bm{\mu}_1 \|_2^2 / n + \sigma_\gamma^2$. In particular, when the latter term is bounded, the assumption $\sigma^2_\varepsilon = \OO(1)$ rules out this degenerate case. On the other hand, if $\SNR$ tends to infinity, $\{Y_t\}$ becomes nearly identical to the underlying trend, making it sufficient to use $\{Y_t\}$ directly as an estimate. This occurs when $\sigma^2_\varepsilon \ll \mathbb{E}\| \bm{\mu}_1 \|_2^2 / n + \sigma_\gamma^2$, i.e., either $\sigma^2_\varepsilon$ approaches zero or $\mathbb{E}\| \bm{\mu}_1 \|_2^2 / n$ or $\sigma_\gamma^2$ diverges. Our assumptions are designed to exclude such extreme cases.

Finally, we note that the i.i.d. assumption on $\{\gamma_t\}$ is not essential. It can be readily relaxed to allow $\{\gamma_t\}$ to follow a stationary process. We will pursue this direction in future work.

\section{Asymptotic properties of the mean squared error}\label{sec_mainresultstechinical}

\subsection{Main results}

We study the MSE of the estimator in (\ref{eq_estimatefinal}). The results follow from the bias-variance decomposition in (\ref{eq_biasandvariancedefinition}) and Proposition \ref{propo_mainprop}. We first prepare some notations. 

 Let $f$ denote the spectral density of $\{\varepsilon_t\}$ that 
\begin{equation}\label{eq_spectraldensity}
f(\omega)=\sum_{h \in \mathbb{Z}}\rho(h)e^{\mathrm{i}h\omega}.
\end{equation}
For \(d\geq 1/2\), we define
\begin{equation}\label{eq_c1c2}
\mathsf{c}_1
:=
\frac{f(0)}{2d\pi}
B\left(\frac{1}{2d},\,2-\frac{1}{2d}\right),
\end{equation}
where \(B(\cdot,\cdot)\) denotes the beta function. For \(d>1/2\), also define
\begin{equation}\label{eq_c1c22}
\mathsf{c}_2
:=
\frac{\tau^2(2d-1)}{2d\pi}
B\left(1-\frac{1}{2d},\,1+\frac{1}{2d}\right).
\end{equation}
For $1\leq j\leq n$, set
\begin{equation}\label{eq_points}
\sfx_j=\pi\frac{n+1-j}{n+1},
\end{equation}
and let
\begin{equation}\label{eq_su}
s(u)=2(1-\cos u).
\end{equation}
Moreover, when $d=1/2$, we define the refined approximation
\begin{equation}\label{eq_boundaryrefined}
\mathcal{R}^{\mathrm{ref}}_n(\nu):=\frac{\tau^2\nu^2}{n\log n}\sum_{j=1}^n\frac{s(\sfx_j)^{1/2}}{(1+\nu s(\sfx_j)^{1/2})^2}
+\frac{1}{n}\sum_{j=1}^n\frac{f(\sfx_j)}{(1+\nu s(\sfx_j)^{1/2})^2}.
\end{equation}

The main result of the paper is summarized as follows. 

\begin{theorem}\label{thm_mainthm}
Suppose Assumption \ref{assum_mainassumption} holds and $f(0)>0$. Recall $\nu$ in (\ref{eq_optimizationproblem}). Then the following statements hold as $n\to\infty$.
\begin{enumerate}
\item[(1).] If $d>1/2$, $\nu\to\infty$, and $\nu^{1/(2d)}/n\to0$, then
\begin{equation}\label{eq_dgeaterthan12}
\frac{\mathbb{E}\|\widehat{\bm{\mu}}-\bm{\mu}\|_2^{2}}{n}=[\sfc_{1}\nu^{-\frac{1}{2d}}+\sfc_{2}(%
\nu^{\frac{1}{2d}}/n)^{2d-1}](1+\oo(1)).  
\end{equation}
\item[(2).] If $d=1/2$, $\nu\to\infty$ and $\nu/n\to0$, then for $\mathcal{R}^{\mathrm{ref}}_n(\nu)$ defined in (\ref{eq_boundaryrefined})
\begin{equation}\label{eq_dhalfrefined}
\frac{\mathbb{E}\|\widehat{\bm{\mu}}-\bm{\mu}\|_2^2}{n}
=\mathcal{R}^{\mathrm{ref}}_n(\nu)+\OO\left(n^{-1}+\frac{\nu}{n\log n}\right).
\end{equation}
If, in addition, when $\log n/(\nu\log\nu)\to0$, then
\begin{equation}\label{eq_dexactly1/2}
\frac{\mathbb{E}\|\widehat{\bm{\mu}}-\bm{\mu}\|_2^{2}}{n}=\frac{\tau^2}{\pi}\frac{\log\nu}{\log n}(1+\oo(1)), 
\end{equation}
where $\tau$ is defined in (\ref{eq_sigmaform}). 
\item[(3).] If $0<d<1/2$ and $\nu\asymp1$, then
\begin{align}\label{eq_dlessthan12}
\frac{\mathbb{E}\|\widehat{\bm{\mu}}-\bm{\mu}\|_2^{2}}{n}
=\Big(
\frac{\tau^2(1-2d)\nu^2}{\pi}
\int_{0}^{\pi}
&\frac{(2-2\cos u)^d}
{[1+\nu(2-2\cos u)^d]^2}\dd u
\notag\\
&
+\frac{1}{\pi}
\int_{0}^{\pi}
\frac{f(u)}
{[1+\nu(2-2\cos u)^d]^2}\dd u
\Big)
\left[1+\oo(1)\right].
\end{align}
\end{enumerate}  
\end{theorem}

Theorem \ref{thm_mainthm} establishes the asymptotic results for the MSE of the estimators. Based on these results, we identify several phase transitions. In particular, minimizing (\ref{eq_optimizationproblem}) yields consistent estimators when $d \geq 1/2$, and inconsistent ones when $d < 1/2$. We elaborate on these results as follows.

First, when $d > 1/2$, the two leading terms in (\ref{eq_dgeaterthan12}) are balanced by choosing
\[
\nu^*=\frac{\sfc_1}{(2d-1)\sfc_2}n^{2d-1}.
\]
This choice yields an MSE of order $\OO(n^{-(2d-1)/(2d)})$. The rate matches the usual nonparametric rate $\OO(n^{-2p/(2p+1)})$ for a deterministic trend of smoothness $p=d-1/2$ \cite{stone1982}. Thus, in this asymptotic sense, the random trend in (\ref{eq_differencestandardform}) has effective smoothness $d-1/2$. Second, when $d = 1/2,$ under Assumption \ref{assum_mainassumption}, we find that $\nu \asymp \log n$, which satisfies the additional condition in part (2) of Theorem \ref{thm_mainthm}. Consequently, the right-hand side of (\ref{eq_dexactly1/2}) also vanishes, but at a slower rate of $\OO(\log \log n / \log n)$ compared to the case when $d > 1/2.$ In addition, when $0 < d < 1/2$, under Assumption \ref{assum_mainassumption}, we find that $\nu \asymp 1$, so the right-hand side of (\ref{eq_dlessthan12}) does not vanish, resulting in an inconsistent estimator.

Before concluding this section, we provide two remarks as follows. 
\begin{remark}
We highlight a connection with the perspective in \cite{wahba1978improper}, which assumes a differential operator of order $d \in \mathbb{N}^+$
\begin{equation}\label{eq_differentialform}
\mu_t = \sum_{j=0}^{d-1} \theta_j t^j + \tau \int_0^t (t - u)^{d-1} \dd \mathsf{B}(u),
\end{equation}
where $\mathsf{B}(u)$ is a standard Brownian motion and each $\theta_j$ is normally distributed. Wahba \cite{wahba1978improper} showed that when $d$ is a positive integer, one can obtain an efficient estimator by minimizing the following objective function
\begin{equation}\label{eq_differentialoptimization}
\sum_{t=1}^n (Y_t - \mu_t)^2 + \nu \int_0^1 \left( \frac{\partial^d \mu_u}{\partial u^d} \right)^2 \dd u.
\end{equation}
It is well known in \cite{wahba1978improper} that minimizing the above objective function yields the smoothing spline estimator when $d \in \mathbb{N}^+$. 

The connection between penalizing finite-order differences in (\ref{eq_optimizationproblem}) and finite-order differentials in (\ref{eq_differentialoptimization}) was first identified in \cite{burman_shumway2009estimation} for $d \in \mathbb{N}^+$. We note that our techniques developed in this paper can potentially be applied to analyze (\ref{eq_differentialform}) and extend the results of \cite{wahba1978improper} by studying the MSE properties for $d > 1/2$ without requiring $d \in \mathbb{N}^+$, provided the stochastic integral is well defined in the Itô sense. However, the differential formulation becomes invalid in the Itô framework when $d \leq 1/2$, in which case the problem must be addressed using tools from fractional or Malliavin calculus. In contrast, the difference operator remains well defined even at $d = 1/2$, although it yields a much slower convergence rate. Since this is beyond the scope of the current paper, we will explore this direction in future work.
\end{remark}

\begin{remark}\label{rem_generalizationweighted}
We note that, in addition to minimizing (\ref{eq_optimizationproblem})—which leads to the estimator in (\ref{eq_estimatefinal})—an alternative approach is to consider the weighted least squares method, as discussed in Section 2.2 of \cite{burman_shumway2009estimation}. Specifically, let $R_\varepsilon := \operatorname{Cov}(\bm{\varepsilon}, \bm{\varepsilon})$. Instead of minimizing (\ref{eq_miminizationreducedform}), one may consider minimizing the following objective
\begin{equation*}
(\bm{Y} - \bm{\mu})^\top R_\varepsilon^{-1} (\bm{Y} - \bm{\mu}) + \nu \bm{\mu}^\top (S^{(r-1)})^\top S^{(r-1)} \bm{\mu}.
\end{equation*}
This yields the following estimator
\begin{equation}\label{eq_estimtorwls}
\widehat{\bm{\mu}}_{\wls} = (R_\varepsilon^{-1} + \nu (S^{(r-1)})^\top S^{(r-1)})^{-1} R_\varepsilon^{-1} \bm{Y}.
\end{equation}  

One can follow and slightly generalize the techniques developed in the proof of Theorem \ref{thm_mainthm} to establish the phase transitions and MSE approximations for $\widehat{\bm{\mu}}_{\wls}$. In fact, under the assumption that ${\varepsilon_t}$ follows an AR process, the consistency and convergence rate of the MSE for $\widehat{\bm{\mu}}_{\wls}$ have been established for $d \geq 1 \in \mathbb{N}^+$ in Theorem 4 of \cite{burman_shumway2009estimation}. We believe that analogues of Theorem \ref{thm_mainthm} can also be derived for $\widehat{\bm{\mu}}_{\wls}$ for all $d > 0$, with distinct transitions occurring at $0 < d < 1/2$, $d = 1/2$, and $d > 1/2$. We plan to pursue this direction in future work.

Finally, in practice, $R_{\varepsilon}$ is unknown and must be estimated using an estimator with a sufficiently fast convergence rate. This can be effectively achieved when $\{\varepsilon_t\}$ follows an ARMA model; see the discussion in Section 2.2 of \cite{burman_shumway2009estimation} for more details.  


\end{remark}

%
%

\subsection{Bias-variance argument: proof of Theorem \ref{thm_mainthm}}
In this section, we present the proof of Theorem \ref{thm_mainthm} along with several key technical components. Our analysis is based on the bias-variance decomposition. Let $\overline{\bm{\mu}} = \mathbb{E}(\widehat{\bm{\mu}} \mid \bm{\mu})$. We begin with the following decomposition
\begin{equation}\label{eq_biasandvariancedefinition}
\frac{\mathbb{E} \| \widehat{\bm{\mu}}-\bm{\mu} \|_2^2}{n}=\frac{\mathbb{E}\| \overline{\bm{\mu}}-\bm{\mu}\|_2^2}{n}+\frac{\mathbb{E}\| \widehat{\bm{\mu}}-\overline{\bm{\mu}}\|_2^2}{n}:=\mathcal{B}+\mathcal{V}.
\end{equation}

The following proposition gives the bias and variance approximations. Its proof will be given in Appendix \ref{appendxi_mainprop}.

\begin{proposition}\label{propo_mainprop}
Under the respective asymptotic regimes in Theorem \ref{thm_mainthm}, the bias and variance terms in (\ref{eq_biasandvariancedefinition}) satisfy
\begin{enumerate}
\item[(1).] For the bias part $\mathcal{B},$ we have that 
\begin{equation*}
\mathcal{B}=
\begin{cases}
\sfc_{2}(\nu^{\frac{1}{2d}}/n)^{2d-1}[1+\oo(1)], & d>\frac{1}{2}; \\
\displaystyle \frac{\tau^2\nu^2}{n\log n}\sum_{j=1}^n\frac{s(\sfx_j)^{1/2}}{\{1+\nu s(\sfx_j)^{1/2}\}^2}+\OO\left(\frac{\nu}{n\log n}\right), & d=\frac{1}{2}; \\
\displaystyle \frac{\tau^2(1-2d)\nu^2}{\pi}\int_{0}^{\pi}\frac{(2-2\cos u)^d}{[1+\nu(2-2\cos u)^d]^2}\dd u[1+\oo(1)], & 0<d<\frac{1}{2}.
\end{cases}
\end{equation*}
\item[(2).] For the variance part $\mathcal{V},$ we have that 
\begin{equation*}
\mathcal{V}=
\begin{cases}
\displaystyle \frac{1}{n}\sum_{j=1}^n\frac{f(\sfx_j)}{\{1+\nu s(\sfx_j)^{1/2}\}^2}+\OO(n^{-1}), & d=\frac{1}{2}; \\
\sfc_{1}\nu^{-1/(2d)}[1+\oo(1)], & d>\frac{1}{2}; \\
\frac{1}{\pi}\int_{0}^{\pi} \frac{f(u)}{[1+\nu(2-2\cos u)^{d}]^{2}}\dd u[1+\oo(1)], & 0<d<\frac{1}{2}.
\end{cases}
\end{equation*} 
\end{enumerate}
\end{proposition}

\begin{proof}[Proof of Theorem \ref{thm_mainthm}]
We apply Proposition \ref{propo_mainprop} separately in the three
regimes.

Suppose first that \(d>1/2\). Proposition
\ref{propo_mainprop} gives
\[
\mathcal{B}
=
\mathsf{c}_2
\left(\frac{\nu^{1/(2d)}}{n}\right)^{2d-1}
[1+\mathrm{o}(1)], \ \ \
\mathcal{V}
=
\mathsf{c}_1\nu^{-1/(2d)}
[1+\mathrm{o}(1)].
\]
Since both leading terms are nonnegative, we can therefore obtain that 
\[
\begin{aligned}
\mathcal{B}+\mathcal{V}
&=
\left[
\mathsf{c}_1\nu^{-1/(2d)}
+
\mathsf{c}_2
\left(\frac{\nu^{1/(2d)}}{n}\right)^{2d-1}
\right]
[1+\mathrm{o}(1)],
\end{aligned}
\]
which proves (\ref{eq_dgeaterthan12}).

Next, let \(d=1/2\). Proposition
\ref{propo_mainprop} yields
\[
\mathcal{B}
=
\frac{\tau^2\nu^2}{n\log n}
\sum_{j=1}^n
\frac{s(\mathsf{x}_j)^{1/2}}
{[1+\nu s(\mathsf{x}_j)^{1/2}]^2}
+
\mathrm{O}\left(\frac{\nu}{n\log n}\right), \ \ \ 
\mathcal{V}
=
\frac{1}{n}
\sum_{j=1}^n
\frac{f(\mathsf{x}_j)}
{[1+\nu s(\mathsf{x}_j)^{1/2}]^2}
+
\mathrm{O}(n^{-1}).
\]
Adding these two expressions proves
(\ref{eq_dhalfrefined}). To show (\ref{eq_dexactly1/2}), by Lemma \ref{lemma_boundary_scaled_sums} and Theorem \ref{thm_maintechinicaltheorem}  of the appendix,  we have that
\[
\mathcal{B}
=
\frac{\tau^2}{\pi}
\frac{\log \nu}{\log n}
[1+\mathrm{o}(1)],
\ \
\mathcal{V}
=
\frac{f(0)}{\pi}\nu^{-1}[1+\mathrm{o}(1)].
\]
Moreover,
\[
\frac{\mathcal{V}}{\mathcal{B}}
=
\mathrm{O}\left(
\frac{\log n}{\nu\log \nu}
\right)
=
\mathrm{o}(1),
\]
under the additional condition
\(\log n/(\nu\log \nu)\to0\). The remainder in
(\ref{eq_dhalfrefined}) is also negligible relative to
\(\log \nu/\log n\), since \(\nu/n\to0\). This establishes (\ref{eq_dexactly1/2}).

Finally, when \(0<d<1/2\), Proposition
\ref{propo_mainprop} gives
\[
\mathcal{B}
=
\frac{\tau^2(1-2d)\nu^2}{\pi}
\int_0^\pi
\frac{(2-2\cos u)^d}
{[1+\nu(2-2\cos u)^d]^2}
\,\mathrm{d}u
[1+\mathrm{o}(1)],
\]
and
\[
\mathcal{V}
=
\frac{1}{\pi}
\int_0^\pi
\frac{f(u)}
{[1+\nu(2-2\cos u)^d]^2}
\,\mathrm{d}u
[1+\mathrm{o}(1)].
\]
Since \(\nu\asymp1\), both leading terms are finite and
nonnegative. This proves (\ref{eq_dlessthan12}).
%
\end{proof}

We now provide a few remarks below.


%
%
\begin{remark}
When $0<d<1/2$, the nontrivial regime under Assumption \ref{assum_mainassumption} is $\nu\asymp1$. In this regime, both terms in (\ref{eq_dlessthan12}) remain non-negligible, so the MSE does not vanish. Thus minimizing (\ref{eq_optimizationproblem}) cannot yield a consistent trend estimator in this range of $d$.

%

%
\end{remark}
\begin{remark}\label{rmk_tecnihcalwls}
In Proposition \ref{propo_mainprop} and Theorem \ref{thm_mainthm}, we analyze the standard MSE, defined as the Euclidean distance between the true random trend $\bm{\mu}$ and its estimate $\widehat{\bm{\mu}}$. 
It is also of interest to consider the \emph{weighted MSE}, i.e., the quadratic form \cite{carroll} 
$
(\widehat{\bm{\mu}}-\bm{\mu})^{\prime}R_{\varepsilon}^{-1}(\widehat{\bm{\mu}}-\bm{\mu}),
$
where $R_{\varepsilon}$ is as in (\ref{eq_estimtorwls}). 
If the spectral density of $\{\varepsilon_t\}$ is bounded away from zero and infinity — which holds, for example, for the usual ARMA model — then the theory of Toeplitz matrices \cite{grenander_szego_1958} implies that all eigenvalues of $R_{\varepsilon}$ are bounded between two positive constants. 
In this case, there exist constants $k_{1}>0$ and $k_{2}>0$ such that 
$$
k_{1}\|\widehat{\bm{\mu}}-\bm{\mu}\|^{2}
\ \leq\ 
(\widehat{\bm{\mu}}-\bm{\mu})^{\prime}R_{\varepsilon}^{-1}(\widehat{\bm{\mu}}-\bm{\mu})
\ \leq\ 
k_{2}\|\widehat{\bm{\mu}}-\bm{\mu}\|^{2}.
$$
Consequently, all asymptotic convergence rate results established for the standard MSE also hold when the distance is measured by the weighted MSE. We will pursue this direction in the future work.   
\end{remark}

\section{ Practical selection of the parameters}\label{sec_selectionofparameters}

In this section, we develop a practical, data-driven rule for selecting the
penalty parameter \(\nu\) and the differencing order \(d\) when
\(d\geq1/2\).  For $d < 1/2$, the signal component is asymptotically
negligible relative to the observation noise, so consistent estimation of
$d$ is not feasible in this setting. The procedure is motivated by the Mallows-type criterion in
Section 3 of \cite{burman_shumway2009estimation}.

Recall the estimator in (\ref{eq_estimatefinal}). For each
\((\nu,d)\in(0,\infty)\times[1/2,\infty)\), let
\begin{equation*}
H_{\nu,d}
=
\left(
I+\nu\bigl(S^{(r-1)}\bigr)^{\top}S^{(r-1)}
\right)^{-1},
\qquad
\widehat{\bm{\mu}}(\nu,d)
=
H_{\nu,d}\bm{Y},
\end{equation*}
where \(r=\lceil d\rceil+1\). Let $\mathcal{G}_n
\subset
(0,\infty)\times[1/2,\infty)$ be a finite candidate grid. In implementation, the candidate orders \(d\)
are chosen on a user-specified grid in \([1/2,d_{\max}]\), and, for each
candidate \(d\), the candidate penalty parameters \(\nu\) are chosen on a
logarithmic grid over a range determined by the sample size \(n\).

Let $\widehat f(u)$ be an estimator of the spectral density of $\{\varepsilon_t\}$, obtained, for example, from a preliminary trend estimate and a fitted short-range-dependent error model.  Such a residual
spectral-density estimator is standard; see, for example,
\cite[Section 3]{burman_shumway2009estimation} and \cite{brockwell2016introduction}. Following \cite[Section 3]{burman_shumway2009estimation}, our primary criterion is
\begin{equation}\label{eq_selectioncriterion}
\phi_n(\nu,d)=\frac{1}{n}\left\|\bm{Y}-\widehat{\bm{\mu}}(\nu,d)\right\|_2^2+\frac{2}{n}\sum_{j=1}^n\frac{\widehat f(\sfx_j)}{1+\nu s(\sfx_j)^d},
\end{equation}
where we recall (\ref{eq_points}) and (\ref{eq_su}). 
We then select
\begin{equation}\label{eq_selectedparameters}
(\widehat\nu,\widehat d)\in\arg\min_{(\nu,d)\in\mathcal{G}_n}\phi_n(\nu,d).
\end{equation}
In our implementation, \(\widehat f\) is obtained by fitting autoregressive
models to preliminary residuals and selecting the order by BIC.

For large $\nu$, the weights in the correction term concentrate near frequency zero. Thus, when $\widehat f(u)$ is sufficiently smooth near zero, the full-spectrum criterion in (\ref{eq_selectioncriterion}) has the simpler approximation
\begin{equation}\label{eq_selectioncriterionzero}
\phi_n^{(0)}(\nu,d)=\frac{1}{n}\left\|\bm{Y}-\widehat{\bm{\mu}}(\nu,d)\right\|_2^2+\frac{2\widehat f(0)}{n}\sum_{j=1}^n\left(1+\nu s(\sfx_j)^d\right)^{-1}.
\end{equation}
The full-spectrum criterion is preferred in implementation; (\ref{eq_selectioncriterionzero}) is useful as a computational simplification and for asymptotic interpretation.

Before concluding this section, we summarize the implementation of the
selection procedure. Let $k=\lceil n^{1/2}\rceil.$ Here, \(k\) is the half-width of the local fitting window, so that the
interior window contains \(2k+1\) observations.

\begin{enumerate}
\item For each \(t=1,\ldots,n\), define
\begin{equation*}
\mathcal{I}_t
=
\left\{
\max(1,t-k),\ldots,\min(n,t+k)
\right\}.
\end{equation*}
Fit a local linear regression to the observations
\(\{Y_j:j\in\mathcal{I}_t\}\) by minimizing
$\sum_{j\in\mathcal{I}_t}
\left\{
Y_j-a-b(j-t)
\right\}^2,$
over \(a\) and \(b\), and let \(\widehat{\mu}_t^{+}\) be the resulting
intercept estimate \(\widehat a\). This automatically provides the
appropriate boundary modification.

\item Form the preliminary residuals
\begin{equation*}
\widetilde{\varepsilon}_t
=
Y_t-\widehat{\mu}_t^{+},
\qquad
t=1,\ldots,n.
\end{equation*}

\item Fit \(\operatorname{AR}(p)\) models to
\(\{\widetilde{\varepsilon}_t\}_{t=1}^n\) over a prespecified collection of
orders. Select the order \(\widehat p\) using BIC, and let
\(\widehat f\) denote the spectral density associated with the fitted
\(\operatorname{AR}(\widehat p)\) model.

\item Evaluate the criterion in (\ref{eq_selectioncriterion}) at every
\((\nu,d)\in\mathcal{G}_n\), and let
\begin{equation*}
(\widehat{\nu},\widehat d)
\in
\arg\min_{(\nu,d)\in\mathcal{G}_n}
\phi_n(\nu,d).
\end{equation*}
Finally, report $\widehat{\bm{\mu}}(\widehat{\nu},\widehat d)$ in (\ref{eq_estimatefinal}).
\end{enumerate}

\section{Numerical studies}\label{sec_numericalstudy}

In this section, we conduct Monte Carlo experiments using the exact finite-sequence fractional operators in \eqref{eq_definitionoperator} and the estimator in \eqref{eq_estimatefinal}. Throughout, we set
\begin{equation*}
\tau^2=\sigma_\varepsilon^2=1.
\end{equation*}
The innovation sequences \(\{\gamma_t\}\) in \eqref{eq_differencestandardform} and \(\{\varepsilon_t\}\) in \eqref{eq_basicmodel} are Gaussian sequences, and \(\sigma_\gamma^2\) is specified according to \eqref{eq_sigmaform}. The penalty parameter is set to its model value \(\nu=\sigma_\varepsilon^2/\sigma_\gamma^2\), as in \eqref{eq_optimizationproblem}.

In what follows, we consider different combinations of \(d\) and \(n\), where \(d\in\{0.25,0.5,0.75,1,1.5\}\) and \(n\in\{400,600,800,1000,1200,1400,1600,1800,2000,2200\}\), and generate \(500\) independent data sets for each combination. In each replication, we construct \(\bm{\mu}=\calS_d\bm{\gamma}\) and compute \(\widehat{\bm{\mu}}\) from \eqref{eq_estimatefinal}. The replication-specific per-coordinate squared error is
\begin{equation*}
\operatorname{MSE}_{n,m}
=
\frac{1}{n}
\left\|
\widehat{\bm{\mu}}^{(m)}-\bm{\mu}^{(m)}
\right\|_2^2,
\qquad
m=1,\ldots,500.
\end{equation*}
The empirical MSE reported below is the average of these \(500\) values.

First, we compare the empirical MSE with the theoretical approximation derived in Theorem \ref{thm_mainthm}. In this experiment, the differencing order, penalty parameter, and error spectrum are known. Thus, the comparison directly assesses the finite-sample accuracy of the asymptotic risk approximation. The results are reported in Figure \ref{fig:phase-transition}.

The theoretical approximations agree closely with the empirical MSEs, and their accuracy generally improves as \(n\) increases, supporting the asymptotic results in Theorem \ref{thm_mainthm}. Moreover, the figure illustrates the phase transition predicted by the theory. When \(d>1/2\), the risk decreases polynomially with \(n\). At the boundary \(d=1/2\), the decrease is substantially slower. When \(d<1/2\), the risk approaches a positive limiting value. Across the displayed sample sizes, the theoretical approximations capture these distinct finite-sample behaviors.

\begin{figure}[!ht]
\centering
\includegraphics[width=3.8in, height=2.5in]{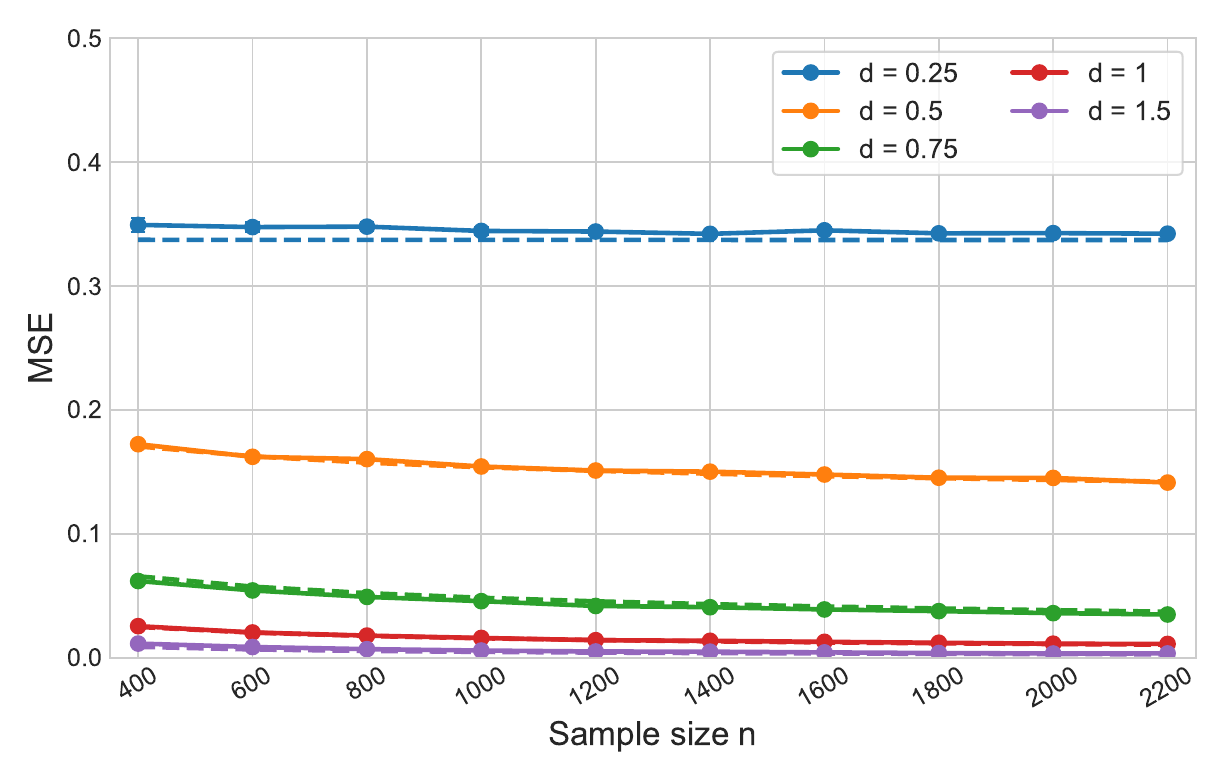}
\caption{Empirical MSE (points) and theoretical approximation (dashes). }
\label{fig:phase-transition}
\end{figure}

Second, we examine the practical selection procedure described in Section \ref{sec_selectionofparameters}. We consider data-generating trend orders \(d\in\{0.5,0.75,1\}\) under three error specifications: white noise; an \(\operatorname{AR}(1)\) process satisfying \(\varepsilon_t=\phi\varepsilon_{t-1}+w_t\); and an \(\operatorname{ARMA}(1,1)\) process satisfying \(\varepsilon_t=\phi \varepsilon_{t-1}+\theta w_{t-1}+w_t\), where \(\{w_t\}\) is a Gaussian white noise sequence. As outlined at the end of Section \ref{sec_selectionofparameters}, for each simulated data set, we first obtain a preliminary local-linear trend estimate. We then fit autoregressive models of orders \(0,\ldots,5\) to the resulting residuals, select the order using BIC, and use the spectral density of the selected autoregressive model as \(\widehat f\) in \eqref{eq_selectioncriterion}. The candidate penalty values are \(\nu\in\{0.5,1,1.5,2,2.5\}\log n\) when \(d=0.5\), and \(\nu\in\{0.5,1,1.5,2,2.5\}n^{2d-1}\) when \(d>0.5\). For each replication, let \((\widehat\nu,\widehat d)\) denote the resulting fitted pair.

We compute the empirical MSE of \(\widehat{\bm{\mu}}(\widehat\nu,\widehat d)\) in \eqref{eq_estimatefinal} and compare it with the theoretical risk approximation evaluated at the same fitted pair. The latter uses \eqref{eq_boundaryrefined} when \(\widehat d=0.5\) and the right-hand side of \eqref{eq_dgeaterthan12} when \(\widehat d>0.5\). The known simulation spectrum is used only to evaluate this theoretical benchmark; the parameter-selection procedure itself uses the residual-based BIC--AR estimate \(\widehat f\).

Figure \ref{fig:selection-plugin} shows that, after parameter selection, the plug-in theoretical approximation continues to track the empirical risk reasonably well for different values of $d \geq 0.5$. This comparison provides a numerical check that the practical residual-based implementation selects penalties in a range for which the risk approximation remains informative. 


\begin{figure}[t]
\centering
\hspace{-0.1in}
\includegraphics[width=1.02\textwidth]{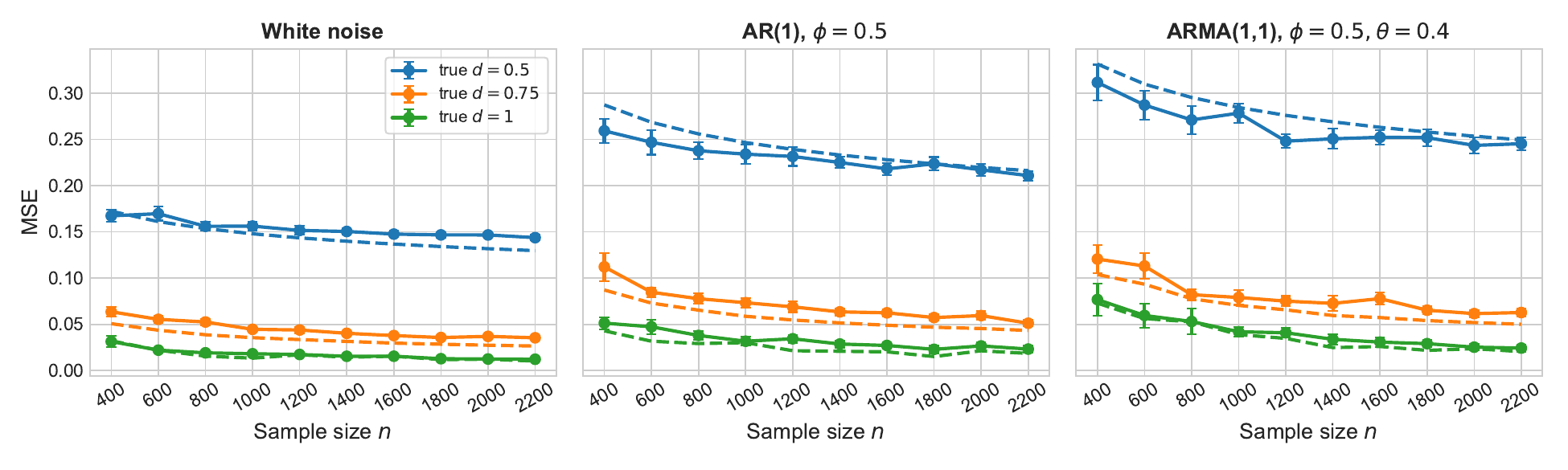}
\caption{Empirical MSE (points) and theoretical approximation evaluated at the fitted pair (dashes).  Error bars are Monte Carlo \(95\%\) confidence intervals.}
\label{fig:selection-plugin}
\end{figure}

\appendix

\section{Proof of Proposition \ref{propo_mainprop}}\label{appendxi_mainprop}

\subsection{Some preliminary results}
In this section, we establish auxiliary results used in the proof of Proposition \ref{propo_mainprop}. The following theorem provides a precise approximation for the bias and variance components in (\ref{eq_biasandvariancedefinition}) whose proof will be given in Appendix \ref{appendix_techinmain}. Recall $\sfx_j$ and $s(\cdot)$ from (\ref{eq_points}) and (\ref{eq_su}).
Recall the spectral density function $f(\cdot)$ in (\ref{eq_spectraldensity}) and the bias-variance decomposition in (\ref{eq_biasandvariancedefinition}). 
\begin{theorem}\label{thm_maintechinicaltheorem}
Suppose the assumptions of Proposition \ref{propo_mainprop} hold. Then we have that for all $\nu>0$
\begin{equation}\label{eq_varaincepart}
 \mathbb{E}\|\widehat{\bm{\mu}}-\overline{\bm{\mu}}\|^{2}=\sum_{j=1}^n (1+\nu
s(\sfx_{j})^{d})^{-2}f(\sfx_{j})+\OO(1),   
\end{equation}
where $\sfx_j$ and $s(\cdot)$ are defined in (\ref{eq_points}) and (\ref{eq_su}), respectively. Moreover, we have that 
\begin{equation}\label{eq_biaspart}
\mathbb{E}\|\overline{\bm{\mu}}-\bm{\mu} \|^{2}=\sigma _{\gamma }^{2}\nu \sum_{j=1}^n
(1+\nu s(\sfx_{j})^{d})^{-2}\nu s(\sfx_{j})^{d} +\OO((\sigma _{\gamma }^{2}\nu )).
\end{equation}
\end{theorem}

The following lemma establishes the asymptotic behavior of the Riemann-type sums in \eqref{eq_varaincepart} and \eqref{eq_biaspart}.

\begin{lemma}
\label{lemma_general_scaled_sum}
Let
\[
h_{\mathrm{V}}(z)=\frac{1}{(1+z)^2},
\qquad
h_{\mathrm{B}}(z)=\frac{z}{(1+z)^2},
\]
and let \(g\) be continuously differentiable on \([0,\pi]\). Define
\begin{equation}\label{eq_someimportantnotations}
a_n=\frac{\pi\nu^{1/(2d)}}{n+1},
\qquad
b_n=\frac{\pi n\nu^{1/(2d)}}{n+1},
\qquad
\psi(x)=\left(\frac{\sin(x/2)}{x/2}\right)^{2d}.
\end{equation}
Then, for every \(d>0\), \(\nu>0\) and
\(h\in\{h_{\mathrm{V}},h_{\mathrm{B}}\}\)
\[
\begin{aligned}
\sum_{j=1}^n h(\nu s(\sfx_j)^d)g(\sfx_j)=
\frac{(n+1)\nu^{-1/(2d)}}{\pi}
\int_{a_n}^{b_n}
h(w^{2d}\psi(\nu^{-1/(2d)}w))
g(\nu^{-1/(2d)}w)\,\mathrm{d}w
+\mathrm{O}(1).
\end{aligned}
\]
\end{lemma}

\begin{proof}
For \(u\in[0,\pi]\), define
\[
F_{\nu}(u)=h(\nu s(u)^d)g(u).
\]
Since
\begin{equation*}
s(u)=2(1-\cos u)=4\sin^2(u/2),
\end{equation*}
we have
\begin{equation*}
s(u)^d
=
\left(4\sin^2(u/2)\right)^d
=
u^{2d}
\left(
\frac{\sin(u/2)}{u/2}
\right)^{2d}.
\end{equation*}
Recall (\ref{eq_someimportantnotations}). This yields that $s(u)^d=u^{2d}\psi(u)$. Consequently, we see that 
\begin{equation}\label{eq_hahahere}
\begin{aligned}
&\frac{n+1}{\pi}
\int_{\pi/(n+1)}^{\pi n/(n+1)}
h(\nu s(u)^d)g(u)\,\mathrm{d}u \\
&\quad=
\frac{(n+1)\nu^{-1/(2d)}}{\pi}
\int_{a_n}^{b_n}
h(w^{2d}\psi(\nu^{-1/(2d)}w))
g(\nu^{-1/(2d)}w)\,\mathrm{d}w,
\end{aligned}
\end{equation}
where the equality follows from the change of variables
\(w=\nu^{1/(2d)}u\).

It remains to justify the replacement of the sum by the integral.
The mesh points \(\sfx_j=\pi(n+1-j)/(n+1)\) have spacing
\(\pi/(n+1)\). We use the following elementary Riemann-sum error bound for functions of
bounded variation (see Chapter 7 of \cite{Apostol1974}). If \(F\) has bounded variation on \([0,\pi]\), then
\begin{equation*}
\left|
\frac{\pi}{n+1}\sum_{j=1}^n F(\sfx_j)
-
\int_{\pi/(n+1)}^{\pi n/(n+1)}F(u)\,\mathrm{d}u
\right|
\leq
\frac{\pi}{n+1}
\operatorname{TV}_{[0,\pi]}(F).
\end{equation*}
Consequently,
\begin{equation*}
\left|
\sum_{j=1}^n F(\sfx_j)
-
\frac{n+1}{\pi}
\int_{\pi/(n+1)}^{\pi n/(n+1)}F(u)\,\mathrm{d}u
\right|
\leq
\operatorname{TV}_{[0,\pi]}(F).
\end{equation*} 
We now verify that \(\operatorname{TV}(F_{\nu})\) is uniformly
bounded. The function \(s\) is increasing on \([0,\pi]\). Moreover,
\(h_{\mathrm{V}}\) is decreasing, while \(h_{\mathrm{B}}\) increases
on \([0,1]\) and decreases on \([1,\infty)\). Hence, for either choice
of \(h\), there exists some constant $C_h>0,$
\[
\operatorname{TV}(h(\nu s(\cdot)^d))\leq C_h,
\]
uniformly in \(\nu>0\). Since \(g\) is continuously differentiable on
\([0,\pi]\), it has bounded variation. The product rule for total
variation (see Chapter 6 of \cite{Apostol1974}) therefore implies
\[
\operatorname{TV}(F_{\nu})\leq C,
\]
where \(C\) does not depend on \(n\) or \(\nu\). Consequently, together with (\ref{eq_hahahere}), we can conclude the proof. 
\end{proof}

\subsection{Technical proof of the Proposition \ref{propo_mainprop}}

Recall (\ref{eq_biasandvariancedefinition}). 
Theorem \ref{thm_maintechinicaltheorem} gives
\[
\mathcal{V}=\frac{1}{n}\sum_{j=1}^n
\frac{f(\sfx_j)}{(1+\nu s(\sfx_j)^d)^2}+\OO(n^{-1})
\]
and
\[
\mathcal{B}=\frac{\sigma_\gamma^2\nu}{n}
\sum_{j=1}^n
\frac{\nu s(\sfx_j)^d}{(1+\nu s(\sfx_j)^d)^2}
+\OO\left(\frac{\sigma_\gamma^2\nu}{n}\right).
\]

For \(d>1/2\), for the variance term, we apply Lemma \ref{lemma_general_scaled_sum}  with
\begin{equation}\label{eq_lemmaa2variancefunction}
h(z)=\frac{1}{(1+z)^2},
\qquad
g(u)=f(u).
\end{equation}
Since \(a_n\downarrow 0\), \(b_n\uparrow\infty\),
\(\psi(\nu^{-1/(2d)}w)\to1\), and
\(f(\nu^{-1/(2d)}w)\to f(0)\) for every fixed \(w\), dominated
convergence gives
\begin{equation*}
\begin{aligned}
\mathcal{V}
&=
\frac{\nu^{-1/(2d)}}{\pi}
\int_{a_n}^{b_n}
\frac{f(\nu^{-1/(2d)}w)}
{(1+w^{2d}\psi(\nu^{-1/(2d)}w))^2}
\,\mathrm{d}w
+\mathrm{O}(n^{-1}) \\
&=
\frac{f(0)\nu^{-1/(2d)}}{\pi}
\int_0^\infty\frac{1}{(1+w^{2d})^2}\,\mathrm{d}w
(1+\mathrm{o}(1))+\mathrm{O}(n^{-1}) \\
&=
\mathsf{c}_1\nu^{-1/(2d)}
(1+\mathrm{o}(1))+\mathrm{O}(n^{-1}).
\end{aligned}
\end{equation*}

Similarly, for the bias term, we apply Lemma \ref{lemma_general_scaled_sum} with
\begin{equation}\label{eq_lemma2functionsbias}
h(z)=\frac{z}{(1+z)^2},
\qquad
g(u)=1.
\end{equation}
Using (\ref{eq_sigmaform}) that
\begin{equation*}
\sigma_\gamma^2
=
\tau^2(2d-1)n^{-(2d-1)}
(1+\mathrm{o}(1)),
\end{equation*}
we obtain
\begin{equation*}
\begin{aligned}
\mathcal{B}
&=
\frac{\sigma_\gamma^2\nu^{1-1/(2d)}}{\pi}
\int_0^\infty
\frac{w^{2d}}{(1+w^{2d})^2}\,\mathrm{d}w
(1+\mathrm{o}(1))+\mathrm{O}(n^{-1}) \\
&=
\mathsf{c}_2
\left(\frac{\nu^{1/(2d)}}{n}\right)^{2d-1}
(1+\mathrm{o}(1))+\mathrm{O}(n^{-1}).
\end{aligned}
\end{equation*}

Because \(\nu^{1/(2d)}/n\to0\), the remainder
\(\mathrm{O}(n^{-1})\) is negligible relative to
\(\nu^{-1/(2d)}\), and the remainder
\(\mathrm{O}(\sigma_\gamma^2\nu/n)\) is negligible relative to
\(\sigma_\gamma^2\nu^{1-1/(2d)}\). Thus, both additive remainders are
absorbed into the respective \(\mathrm{o}(1)\) terms. This concludes the proof for $d>1/2.$

For \(d=1/2\), the identity
\(\sigma_\gamma^2=\tau^2/\log n\) gives
\[
\mathcal{B}
=\frac{\tau^2\nu^2}{n\log n}
\sum_{j=1}^n
\frac{s(\sfx_j)^{1/2}}
{(1+\nu s(\sfx_j)^{1/2})^2}
+\OO\left(\frac{\nu}{n\log n}\right).
\]
Likewise, we have 
\[
\mathcal{V}
=\frac{1}{n}\sum_{j=1}^n
\frac{f(\sfx_j)}
{(1+\nu s(\sfx_j)^{1/2})^2}
+\OO(n^{-1}).
\]
These are precisely the bias and variance expressions in the
proposition. The further asymptotic simplifications of these sums are
recorded separately in Lemma \ref{lemma_boundary_scaled_sums}.

Finally, let \(0<d<1/2\), and define
\[
h_\nu(u)=\frac{s(u)^d}{(1+\nu s(u)^d)^2},
\qquad
v_\nu(u)=\frac{f(u)}{(1+\nu s(u)^d)^2}.
\]
Since \(\nu\asymp1\), the functions \(h_\nu\) and \(v_\nu\) are uniformly
continuously differentiable on \([0,\pi]\). By the first-order
Euler--Maclaurin formula (see Chapter 7 of \cite{Apostol1974}), we have that 
\begin{equation*}
\frac{1}{n}\sum_{j=1}^n h_\nu(\sfx_j)
=
\frac{1}{\pi}\int_0^\pi h_\nu(u)\,\mathrm{d}u
+\mathrm{o}(1),
\end{equation*}
and
\begin{equation*}
\frac{1}{n}\sum_{j=1}^n v_\nu(\sfx_j)
=
\frac{1}{\pi}\int_0^\pi v_\nu(u)\,\mathrm{d}u
+\mathrm{o}(1).
\end{equation*}
Since \(\sigma_\gamma^2=\tau^2(1-2d)(1+\oo(1))\) according to (\ref{eq_sigmaform}), we obtain from (\ref{eq_biaspart}) that
\[
\mathcal{B}
=\frac{\tau^2(1-2d)\nu^2}{\pi}
\int_0^\pi
\frac{s(u)^d}{(1+\nu s(u)^d)^2}\,\mathrm{d}u
(1+\oo(1)),
\]
where we used the fact that the remainder \(\sigma_\gamma^2\nu/n\) is \(\mathrm{o}(1)\). Similarly,
\[
\mathcal{V}
=\frac{1}{\pi}\int_0^\pi
\frac{f(u)}{(1+\nu s(u)^d)^2}\,\mathrm{d}u
(1+\oo(1)).
\]
This proves the \(0<d<1/2\) part and completes the proof of Proposition \ref{propo_mainprop}.

\subsection{Boundary asymptotics at \(d=1/2\)}
The case \(d=1/2\) is the boundary between the consistent and inconsistent
regimes and requires a separate asymptotic argument. We record the required
scaled-sum approximations below which is used in the proof of Theorem \ref{thm_mainthm}.
\begin{lemma}\label{lemma_boundary_scaled_sums}
Suppose Assumption \ref{assum_mainassumption} holds and  \(d=1/2\), \(\nu\to\infty\), and \(\nu/n\to0\). Then
\[
\frac{1}{n}\sum_{j=1}^n
\frac{f(\sfx_j)}{(1+\nu s(\sfx_j)^{1/2})^2}
=\frac{f(0)}{\pi\nu}(1+\oo(1)),
\]
and
\[
\frac{\nu^2}{n}\sum_{j=1}^n
\frac{s(\sfx_j)^{1/2}}
{(1+\nu s(\sfx_j)^{1/2})^2}
=\frac{\log\nu}{\pi}(1+\oo(1)).
\]
\end{lemma}

\begin{proof}
Recall (\ref{eq_someimportantnotations}).
For \(d=1/2\), we have 
\begin{equation*}
\psi(u)=\frac{\sin(u/2)}{u/2},
\qquad
a_n=\frac{\pi\nu}{n+1},
\qquad
b_n=\frac{\pi n\nu}{n+1},
\end{equation*}
and $s(u)^{1/2}=u\psi(u).$

We first consider the variance sum. Apply Lemma
\ref{lemma_general_scaled_sum} with (\ref{eq_lemmaa2variancefunction}), we have that 
\begin{equation*}
\begin{aligned}
&\frac{1}{n}\sum_{j=1}^n
\frac{f(\sfx_j)}
{(1+\nu s(\sfx_j)^{1/2})^2} \\
&\quad=
\frac{n+1}{n\pi\nu}
\int_{a_n}^{b_n}
\frac{f(\nu^{-1}w)}
{(1+w\psi(\nu^{-1}w))^2}
\,\mathrm{d}w
+\mathrm{O}(n^{-1}).
\end{aligned}
\end{equation*}
Since \(\nu/n\to0\) and \(\nu\to\infty\), we have that $a_n\longrightarrow0, \ b_n\longrightarrow\infty.$ Moreover, for every fixed \(w>0\),
\begin{equation*}
\psi(\nu^{-1}w)\longrightarrow1,
\qquad
f(\nu^{-1}w)\longrightarrow f(0).
\end{equation*}
Because \(\psi\) is continuous and strictly positive on \([0,\pi]\), there
exists a constant \(c>0\) such that
\begin{equation*}
\psi(u)\geq c,
\qquad
0\leq u\leq\pi.
\end{equation*}
Therefore,
\begin{equation*}
0\leq
\frac{f(\nu^{-1}w)}
{(1+w\psi(\nu^{-1}w))^2}
\leq
\frac{\|f\|_\infty}{(1+cw)^2}.
\end{equation*}
The dominating function is integrable on \([0,\infty)\). Hence, by dominated
convergence, we have
\begin{equation}\label{eq_exampleone}
\begin{aligned}
&\int_{a_n}^{b_n}
\frac{f(\nu^{-1}w)}
{(1+w\psi(\nu^{-1}w))^2}
\,\mathrm{d}w \\
&\quad=
f(0)\int_0^\infty\frac{1}{(1+w)^2}\,\mathrm{d}w
(1+\mathrm{o}(1)) \\
&\quad=
f(0)(1+\mathrm{o}(1)).
\end{aligned}
\end{equation}
It follows that
\begin{equation*}
\frac{1}{n}\sum_{j=1}^n
\frac{f(\sfx_j)}
{(1+\nu s(\sfx_j)^{1/2})^2}
=
\frac{f(0)}{\pi\nu}(1+\mathrm{o}(1)).
\end{equation*}

For the bias part, applying Lemma \ref{lemma_general_scaled_sum} with (\ref{eq_lemma2functionsbias}) and multiplying the resulting identity by \(\nu/n\), we obtain
\begin{equation*}
\begin{aligned}
&\frac{\nu^2}{n}\sum_{j=1}^n
\frac{s(\sfx_j)^{1/2}}
{(1+\nu s(\sfx_j)^{1/2})^2} \\
&\quad=
\frac{n+1}{n\pi}
\int_{a_n}^{b_n}
\frac{w\psi(\nu^{-1}w)}
{(1+w\psi(\nu^{-1}w))^2}
\,\mathrm{d}w
+\mathrm{O}\left(\frac{\nu}{n}\right).
\end{aligned}
\end{equation*}
Since \(\nu/n\to0\), it remains to evaluate the integral. By an argument similar to (\ref{eq_exampleone}), we have 
\begin{equation*}
\int_{a_n}^{b_n}
\frac{w\psi(\nu^{-1}w)}
{(1+w\psi(\nu^{-1}w))^2}
\,\mathrm{d}w
=
\int_{a_n}^{b_n}
\frac{w}{(1+w)^2}\,\mathrm{d}w
+\mathrm{O}(1).
\end{equation*}
Finally, by direct calculation, we have 
\begin{equation*}
\begin{aligned}
\int_{a_n}^{b_n}\frac{w}{(1+w)^2}\,\mathrm{d}w
&=
\log(1+b_n)
+\frac{1}{1+b_n}
-\log(1+a_n)
-\frac{1}{1+a_n} \\
&=
\log\nu+\mathrm{O}(1).
\end{aligned}
\end{equation*}
Consequently, as $\nu \uparrow \infty,$ we have 
\begin{equation*}
\frac{\nu^2}{n}\sum_{j=1}^n
\frac{s(\sfx_j)^{1/2}}
{(1+\nu s(\sfx_j)^{1/2})^2}
=
\frac{\log\nu}{\pi}(1+\mathrm{o}(1)).
\end{equation*}
This completes the proof.
\end{proof}

\section{Proof of Theorem \ref{thm_maintechinicaltheorem}}\label{appendix_technical_lemmas}

\subsection{Technical preparation}

In this section, we collect and prove auxiliary results that will be used in the proof of Theorem \ref{thm_maintechinicaltheorem}. We first collect some results from linear algebra. For a $n \times
 n$ symmetric matrix $A$, we denote 
its eigenvalues by $\lambda_{1}(A)\geq\cdot\cdot\cdot\geq\lambda_{n}(A),$ and its operator norm by $\|A \|.$

\begin{lemma}\label{thm_aux_interlacing}
\noindent Let $A$ and $B$ be two real symmetric matrices so that $A-B$ has rank at most $r$. Then for $r+1\leq j\leq n$ the following inequalities hold:  (a) $\lambda_{j}(A)\leq\lambda_{j-r}(B)$, and (b) $\lambda
_{j}(B)\leq\lambda_{j-r}(A)$.
\end{lemma}
\begin{proof}
See \cite[Theorem 4.3.6]{horn_johnson1985matrix}. 
\end{proof}

\begin{lemma}\label{lem_aux_trace_rank}
Let $A$ and $B$ be two $n \times n$ nonnegative definite matrices with the
property $\operatorname{rank}(A-B)\leq r$. Assume that all the eigenvalues of $A$ and $B$
are bounded above by a constant $c>0$, and the values of the
integer $r$ and the constant $c$ do not depend on $n$. Then we have
\[
\sum_{1\leq j\leq n}\lambda_{j}(A)-\sum_{1\leq j\leq n}\lambda_{j}(B)=\mathrm{O}(1). 
\]
\end{lemma}
\begin{proof}
The result follows directly from Lemma \ref{thm_aux_interlacing}. 
\end{proof}

We next collect and prove auxiliary results concerning the matrix \(S\) introduced below \eqref{eq_operatorproperty2}. The analysis relies on its connection with Toeplitz and Hankel matrices. Recall that a matrix \(T=(b_{jk})\) is Toeplitz if \(b_{jk}=b_{j-k}\). If \(b_{j-k}\) is given by \(\int_{-\pi}^{\pi}\exp\{\mathrm{i}(j-k)u\}f(u)\,\mathrm{d}u/(2\pi)\), then \(f\) is called the symbol of \(T\). We write \(T(f)\) for the corresponding infinite-dimensional Toeplitz matrix and \(T_n(f)\) for its leading \(n\times n\) principal submatrix.

The finite-sample boundary correction arising in the analysis of \(S^\top S\) has Hankel structure. Accordingly, a matrix \(H=(b_{jk})\) is called a Hankel matrix if \(b_{jk}=b_{j+k}\). If \(b_{j+k}\) is given by \(\int_{-\pi}^{\pi}\exp\{\mathrm{i}(j+k)u\}f(u)\,\mathrm{d}u/(2\pi)\), then \(f\) is called the symbol of \(H\). We write \(H(f)\) for the corresponding infinite-dimensional Hankel matrix and \(H_n(f)\) for its leading \(n\times n\) principal submatrix.

With the above preparation, now we connect $S$ with the Toeplitz matrix. For \(d>0\), recall that the \((t,j)\)th entry of the matrix \(S\) is given by
\begin{equation*}
S(t,j)=\left\{
\begin{array}{cc}
(-1)^{t-j}\binom{d}{t-j}, & 1\leq j\leq t, \\
0, & j>t.
\end{array}
\right.
\end{equation*}
Equivalently, according to the generalized binomial expansion, we can write
\begin{equation*}
S(t,j)
=
\frac{1}{2\pi}
\int_{-\pi}^{\pi}
\exp(\mathrm{i}(t-j)u)
(1-\exp(-\mathrm{i}u))^d
\,\mathrm{d}u.
\end{equation*}
More precisely, let \(T(s_0^d)\) denote the infinite lower-triangular Toeplitz operator with symbol \(s_0(u)^d\), where $s_0(u)=1-\exp(-\mathrm{i}u)$. Then \(S\) is the leading \(n\times n\) principal submatrix of \(T(s_0^d)\). Recall from \eqref{eq_su} that \(s(u)=|s_0(u)|^2\). Although the corresponding infinite product has symbol \(s(u)^d\), the finite-dimensional matrix \(S^\top S\) is not exactly \(T_n(s^d)\), because truncation at the sample boundary produces a Hankel-type correction. We formally establish the approximation results as follows. 

To state the result, let \(\mathcal{W}_n\) denote the flip operator defined by
\begin{equation*}
\mathcal{W}_n(z_1,\ldots,z_n)^\top=(z_n,\ldots,z_1)^\top,
\end{equation*}
for every \(\bm{z}=(z_1,\ldots,z_n)^\top\in\mathbb{R}^n\).

\begin{lemma}\label{lem_aux_toeplitz}
Let \(T_n(g)\) be the Toeplitz matrix with symbol $g(u)=\sum_{j\in\mathbb{Z}}g_j\exp(-\mathrm{i}ju),$
where \(g_j=g_{-j}\) for every \(j\in\mathbb{Z}\) and $\sum_{j\in\mathbb{Z}}|jg_j|<\infty$. Define
\begin{equation*}
g_n(u)=\sum_{|j|\leq n}g_j\exp(-\mathrm{i}ju),
\end{equation*}
and let \(R_n(g)=(R_n(j,k;g))_{j,k=1}^n\), where we recall (\ref{eq_points}) and denote
\begin{align*}
R_n(j,k;g)
&=
\frac{2}{n+1}
\sum_{t=1}^{n}
g(\sfx_t)\sin(j\sfx_t)\sin(k\sfx_t), \ \,
1\leq j,k\leq n. 
\end{align*}
Then we have that 
\begin{equation*}
\left\|
T_n(g)-R_n(g)-H_n(g)-\mathcal{W}_nH_n(g)\mathcal{W}_n
\right\|
=
\mathrm{O}\!\left(\left\|g_n-g\right\|_\infty\right).
\end{equation*}
\end{lemma}

\begin{proof}
%
%
First, suppose that \(j+k\leq n\). By the definitions of the Toeplitz and Hankel entries and the identity \(\cos((j-k)u)-\cos((j+k)u)=2\sin(ju)\sin(ku)\), we have
\begin{align*}
T_n(j-k;g_n)-H_n(j+k;g_n)
&=
\frac{2}{\pi}
\int_0^\pi
\sin(ju)\sin(ku)g_n(u)\,\mathrm{d}u \\
&=
\frac{2}{n+1}
\sum_{t=1}^{n}
\sin(j\sfx_t)\sin(k\sfx_t)g_n(\sfx_t) \\
&=
R_n(j,k;g_n).
\end{align*}
Indeed, the second equality is the discrete sine-transform quadrature formula. Since \(g_n\) is an even trigonometric polynomial of degree at most \(n\), this formula follows by expanding \(g_n\) in its Fourier series and applying discrete cosine orthogonality on the grid \(\{\sfx_t\}_{t=1}^n\). The condition \(j+k\leq n\) ensures that no aliasing terms arise in this calculation.

Next, suppose that \(n+1\leq j+k\leq 2n\). Applying the same discrete sine-transform identity to the indices \(n+1-j\) and \(n+1-k\) gives
\begin{align*}
&T_n(j-k;g_n)-H_n(2n+2-j-k;g_n) \\
&\qquad =
\frac{2}{\pi}
\int_0^\pi
\sin ((n+1-j)u)
\sin ((n+1-k)u)
g_n(u)\,\mathrm{d}u \\
&\qquad =
\frac{2}{n+1}
\sum_{t=1}^{n}
\sin ((n+1-j)\sfx_t)
\sin ((n+1-k)\sfx_t)
g_n(\sfx_t) \\
&\qquad =
\frac{2}{n+1}
\sum_{t=1}^{n}
\sin(j\sfx_t)\sin(k\sfx_t)g_n(\sfx_t) \\
&\qquad =
R_n(j,k;g_n).
\end{align*}
For the penultimate equality, note that the two sine factors acquire the same sign under the replacement \(j\mapsto n+1-j\) and \(k\mapsto n+1-k\), so their product is unchanged.

Because \(g_n\) has Fourier coefficients equal to zero outside \(\{-n,\ldots,n\}\), \(H_n(j+k;g_n)=0\) whenever \(j+k>n\). Moreover, the \((j,k)\)th entry of \(\mathcal{W}_nH_n(g_n)\mathcal{W}_n\) is \(H_n(2n+2-j-k;g_n)\), which is zero whenever \(j+k\leq n+1\). Thus, the first identity applies when \(j+k\leq n\), both Hankel terms vanish when \(j+k=n+1\), and the second identity applies when \(j+k\geq n+2\). Combining these cases entrywise yields
\begin{equation*}
T_n(g_n)-R_n(g_n)
=
H_n(g_n)+\mathcal{W}_nH_n(g_n)\mathcal{W}_n.
\end{equation*}

It remains to replace \(g_n\) by \(g\). Put \(h=g-g_n\). For the Hankel term, \(H_n(h)\) is a finite compression of the Hankel operator associated with the symbol \(h\). This operator is obtained by composing multiplication by \(h\), which has operator norm \(\|h\|_\infty\), with an orthogonal projection and a reflection, both of which have operator norm one. Therefore,
\begin{equation*}
\left\|H_n(g)-H_n(g_n)\right\|
=
\left\|H_n(h)\right\|
\leq
\left\|h\right\|_\infty
=
\left\|g-g_n\right\|_\infty.
\end{equation*}

For the sine-basis term, let \(U_n\) be the orthogonal discrete sine-transform matrix with \((j,t)\)th entry \((2/(n+1))^{1/2}\sin(j\sfx_t)\). By the definition of \(R_n(\cdot)\), \(R_n(h)\) is obtained by conjugating the diagonal matrix with diagonal entries \(h(\sfx_1),\ldots,h(\sfx_n)\) by \(U_n\). Orthogonal conjugation preserves the operator norm. Hence,
\begin{equation*}
\left\|R_n(g)-R_n(g_n)\right\|
=
\left\|R_n(h)\right\|
\leq
\max_{1\leq t\leq n}|h(\sfx_t)|
\leq
\left\|g-g_n\right\|_\infty.
\end{equation*}

Moreover, \(\mathcal{W}_n\) is an orthogonal and symmetric matrix; that is, \(\mathcal{W}_n^\top=\mathcal{W}_n=\mathcal{W}_n^{-1}\). Therefore, conjugation by \(\mathcal{W}_n\) preserves the operator norm, and hence
\begin{equation*}
\left\|
\mathcal{W}_n
\bigl (H_n(g)-H_n(g_n)\bigr)
\mathcal{W}_n
\right\|
=
\left\|
H_n(g)-H_n(g_n)
\right\|.
\end{equation*}
Furthermore, the \((j,k)\)th entry of \(T_n(g)\) is \(g_{j-k}\), where \(1\leq j,k\leq n\). Thus, only Fourier coefficients \(g_\ell\) with \(|\ell|\leq n-1\) enter \(T_n(g)\). Since \(g_n\) retains all Fourier coefficients of \(g\) with \(|\ell|\leq n\), the matrices \(T_n(g)\) and \(T_n(g_n)\) have identical entries. Therefore,
\begin{equation*}
T_n(g)=T_n(g_n).
\end{equation*}
Combining this equality with the identity established above for \(g_n\), the difference in the statement of the lemma is the sum of the differences between the corresponding \(R_n\) and \(H_n\) terms. The triangle inequality, together with the preceding bounds, shows that its operator norm is bounded by \(3\|g-g_n\|_\infty\). This proves the result.
\end{proof}

The following results provide the key asymptotic expressions used in the proof of Theorem \ref{thm_maintechinicaltheorem}.
\begin{lemma}\label{thm_aux_trace_approx}
Let \(d>0\) and \(\nu>0\). Suppose that \(g:[0,\pi]\to\mathbb{R}\) is of bounded variation, and let \(R_n(g)\) be defined as in Lemma \ref{lem_aux_toeplitz}. Then, as \(n\to\infty\),
\begin{equation*}
\operatorname{tr}\!\left(
(I+\nu S^\top S)^{-2}R_n(g)
\right)
=
\sum_{j=1}^{n}
\frac{g(\sfx_j)}
{(1+\nu s(\sfx_j)^d)^{2}}
+
\mathrm{O}(1),
\end{equation*}
and
\begin{equation*}
\operatorname{tr}\!\left(
\nu S^\top S
(I+\nu S^\top S)^{-2}
\right)
=
\sum_{j=1}^{n}
\frac{\nu s(\sfx_j)^d}
{(1+\nu s(\sfx_j)^d)^{2}}
+
\mathrm{O}(1).
\end{equation*}
\end{lemma}

\begin{proof}
It suffices to prove the first assertion when \(g\) is bounded, nonnegative, and decreasing. The result for a general function of bounded variation then follows by applying the Jordan decomposition and using linearity. Indeed, let
\[
V_g(u)=\operatorname{TV}_{[0,u]}(g),
\qquad 0\leq u\leq\pi,
\]
and define
\[
g_2(u)=V_g(\pi)-V_g(u)+\|g\|_\infty,
\qquad
g_1(u)=g(u)+g_2(u).
\]
Then \(g_1\) and \(g_2\) are bounded, nonnegative, and decreasing on \([0,\pi]\), and \(g=g_1-g_2\). Since both \(R_n(\cdot)\) and the right-hand side of the asserted approximation are linear in \(g\), the general result follows by applying the decreasing-case result separately to \(g_1\) and \(g_2\).

Write \(A=S^\top S\), and let \(d=\ell+\beta\), where \(\ell\) is a nonnegative integer and \(0<\beta\leq1\). We invoke the finite-rank comparison for fractional-difference operators established in \cite{BURMAN2006677}. In the notation used here, this result implies that, for every \(p\in(0,\pi)\), \(A\) can be compared with
\begin{equation*}
(1-\beta)s(p)^\beta I
+
\beta s(p)^{\beta-1}T_n(s)^{\ell+1},
\end{equation*}
up to a symmetric perturbation of rank at most \(2\ell\), where we recall (\ref{eq_su}). 

We first prove the approximation involving \(R_n(g)\). The matrix \(T_n(s)\) is diagonalized by the discrete sine basis \(\{\bm{e}_t\}_{t=1}^n\) whose entries are $\{\sqrt{2/(n+1)} \sin(j \sfx_t)\}$, with eigenvalues \(s(\sfx_t)\) \cite{grenander_szego_1958}. Since \(R_n(g)\) is diagonalized by the same basis, its eigenvalues are \(g(\sfx_t)\). For a fixed \(p\in(0,\pi)\), the comparison from \cite[Theorem~6(a) and Theorem~7(b), (d)]{BURMAN2006677}, followed by the interlacing theorem in Theorem \ref{thm_aux_interlacing}, gives the lower bound
\begin{align*}
&\lambda_{j+2\ell}\!\left(
(I+\nu A)^{-1}
R_n(g)
(I+\nu A)^{-1}
\right) \\
&\qquad \geq
\left[
1+\nu\left(
(1-\beta)s(p)^\beta
+
\beta s(p)^{\beta-1}s(\sfx_{n+1-j})^{\ell+1}
\right)
\right]^{-2}
g(\sfx_{n+1-j}),
\end{align*}
for \(j=1,\ldots,n-2\ell\).

To identify the right-hand side, choose \(p=\sfx_{n+1-j}\). Since \(d=\ell+\beta\), we have
\begin{equation*}
(1-\beta)s(p)^\beta
+
\beta s(p)^{\beta-1}s(p)^{\ell+1}
=
s(p)^d.
\end{equation*}
Moreover, \(\{\sfx_t\}_{t=1}^n\) is decreasing in \(t\), \(s(\sfx_t)\) is decreasing in \(t\), and \(g(\sfx_t)\) is increasing in \(t\), because \(g\) is decreasing on \([0,\pi]\). Thus, the relevant eigenvalues are ordered as required for the interlacing argument. Summing the preceding inequality over \(j=1,\ldots,n-2\ell\) gives
\begin{align*}
&\operatorname{tr}\!\left(
(I+\nu A)^{-2}R_n(g)
\right) \\
&\qquad \geq
\sum_{j=1}^{n-2\ell}
\frac{g(\sfx_{n+1-j})}
{(1+\nu s(\sfx_{n+1-j})^d)^{2}} \\
&\qquad =
\sum_{j=1}^{n}
\frac{g(\sfx_j)}
{(1+\nu s(\sfx_j)^d)^{2}}
+
\mathrm{O}(1).
\end{align*}
The omitted \(2\ell\) terms are uniformly bounded because \(g\) is bounded. The companion comparison for \(SS^\top\) in \cite[Theorem~6(a) and Theorem~7(b), (d)]{BURMAN2006677}, together with the same interlacing argument, yields the reverse bound
\begin{equation*}
\operatorname{tr}\!\left(
(I+\nu A)^{-2}R_n(g)
\right)
\leq
\sum_{j=1}^{n}
\frac{g(\sfx_j)}
{(1+\nu s(\sfx_j)^d)^{2}}
+
\mathrm{O}(1).
\end{equation*}
This proves the first assertion.

It remains to prove the bias-trace approximation. Since \(R_n(1)=I\), the same argument, with \(g\equiv1\), applies to the two resolvent powers \(q=1\) and \(q=2\), and gives
\begin{equation*}
\operatorname{tr}\!\left(
(I+\nu A)^{-q}
\right)
=
\sum_{j=1}^{n}
(1+\nu s(\sfx_j)^d)^{-q}
+
\mathrm{O}(1),
\qquad
q=1,2.
\end{equation*}
Finally, the elementary identity
\begin{equation*}
\nu A(I+\nu A)^{-2}
=
(I+\nu A)^{-1}
-
(I+\nu A)^{-2},
\end{equation*}
implies that
\begin{align*}
\operatorname{tr}\!\left(
\nu A(I+\nu A)^{-2}
\right)
&=
\sum_{j=1}^{n}
\left[
\frac{1}{1+\nu s(\sfx_j)^d}
-
\frac{1}{(1+\nu s(\sfx_j)^d)^{2}}
\right]
+
\mathrm{O}(1) \\
&=
\sum_{j=1}^{n}
\frac{\nu s(\sfx_j)^d}
{(1+\nu s(\sfx_j)^d)^{2}}
+
\mathrm{O}(1).
\end{align*}
This proves the second assertion.
\end{proof}

\subsection{Proof of Theorem \ref{thm_maintechinicaltheorem}}\label{appendix_techinmain}

For notational simplicity, let $\overline S=S^{(r-1)}$
be the matrix obtained from \(S\) by deleting its first \(r-1\) rows, and define
\begin{equation*}
\overline U_d=\overline S^\top\overline S,
\qquad
U_d=S^\top S.
\end{equation*}
Both \(\overline U_d\) and \(U_d\) are \(n\times n\) symmetric matrices. Throughout this proof, \(f\) denotes the spectral density in \eqref{eq_spectraldensity}. Recall (\ref{eq_estimatefinal}). We can write
\begin{equation*}
H_{\nu,d}=(I+\nu\overline U_d)^{-1},
\qquad
\widehat{\bm{\mu}}=H_{\nu,d}\bm{Y},
\qquad
\overline{\bm{\mu}}=H_{\nu,d}\bm{\mu}.
\end{equation*}

Since \(\widehat{\bm{\mu}}-\overline{\bm{\mu}}=H_{\nu,d}\bm{\varepsilon}\) and \(\mathbb{E}(\bm{\varepsilon}\bm{\varepsilon}^\top)=T_n(f)\), the variance component is
\begin{equation*}
\mathbb{E}\left\|
\widehat{\bm{\mu}}-\overline{\bm{\mu}}
\right\|_2^2
=
\operatorname{tr}\!\left(
H_{\nu,d}^2T_n(f)
\right).
\end{equation*}
The matrices \(\overline U_d\) and \(U_d\) differ by a matrix of rank at most \(r-1\). Indeed, the resolvent identity gives
\begin{equation*}
\mathit{H}_{\nu,d}-(I+\nu U_d)^{-1}
=
\nu\mathit{H}_{\nu,d}(U_d-\overline U_d)(I+\nu U_d)^{-1}.
\end{equation*}
Hence, this difference has rank at most \(r-1\). Consequently,
\(\mathit{H}_{\nu,d}^2-(I+\nu U_d)^{-2}\) has rank at most \(2(r-1)\). Both resolvents are contractions because \(\overline U_d\) and \(U_d\) are positive semidefinite, while \(\|T_n(f)\|\leq\|f\|_\infty\) because \(f\) is bounded \cite{grenander_szego_1958}. Thus, the trace of the resulting bounded-rank perturbation is \(\mathrm{O}(1)\), and hence
\begin{equation*}
\operatorname{tr}\!\left(
H_{\nu,d}^2T_n(f)
\right)
=
\operatorname{tr}\!\left(
(I+\nu U_d)^{-2}T_n(f)
\right)
+
\mathrm{O}(1).
\end{equation*}
Let $B_\nu=(I+\nu U_d)^{-2}.$ By Lemma \ref{lem_aux_toeplitz}, we may write
\[
T_n(f)-R_n(f)
=
H_n(f)+\mathcal{W}_nH_n(f)\mathcal{W}_n+E_n,
\]
where
\[
\|E_n\|
=
\mathrm{O}\!\left(\|f-f_n\|_\infty\right).
\]
Assumption \ref{assum_mainassumption}(1) implies that
\(\sum_{m\in\mathbb{Z}}|m f_m|<\infty\). To bound the Hankel term, write
\[
H_n(f)=\sum_{m=2}^{2n}f_mJ_m,
\]
where \(J_m\) has ones on its \(m\)th anti-diagonal and zeros elsewhere.
Since \(\|J_m\|_1\leq m-1\), we have
\[
\|H_n(f)\|_1
\leq
\sum_{m=2}^{2n}(m-1)|f_m|
\leq
\sum_{m\in\mathbb{Z}}|m f_m|
=
\mathrm{O}(1).
\]
Because \(\mathcal{W}_n\) is orthogonal,
\[
\|\mathcal{W}_nH_n(f)\mathcal{W}_n\|_1
=
\|H_n(f)\|_1
=
\mathrm{O}(1).
\]
Moreover,
\[
n\|f-f_n\|_\infty
\leq
n\sum_{|m|>n}|f_m|
\leq
\sum_{|m|>n}|m f_m|
=
\mathrm{o}(1).
\]

Since \(\|B_\nu\|\leq1\), the trace inequality
\(\left|\operatorname{tr}(B_\nu C)\right|\leq\|B_\nu\|\|C\|_1\)
shows that each Hankel term contributes \(\mathrm{O}(1)\). Also,
\[
\left|\operatorname{tr}(B_\nu E_n)\right|
\leq
n\|B_\nu\|\|E_n\|
=
\mathrm{o}(1).
\]
Consequently,
\[
\operatorname{tr}\!\left(B_\nu T_n(f)\right)
=
\operatorname{tr}\!\left(B_\nu R_n(f)\right)
+
\mathrm{O}(1).
\]

Furthermore, (1) of Assumption \ref{assum_mainassumption} implies
\(\sum_{h\in\mathbb{Z}}|h\rho(h)|<\infty\). Hence the spectral density \(f\) is continuously differentiable on \([0,\pi]\), and therefore has bounded variation. Thus, the first assertion of Lemma \ref{thm_aux_trace_approx} applies with \(g=f\) and gives
\begin{equation*}
\mathbb{E}\left\|
\widehat{\bm{\mu}}-\overline{\bm{\mu}}
\right\|_2^2
=
\sum_{j=1}^{n}
\frac{f(\sfx_j)}
{(1+\nu s(\sfx_j)^d)^{2}}
+
\mathrm{O}(1),
\end{equation*}
which proves \eqref{eq_varaincepart}.

We next consider the bias component. Recall (\ref{eq_mu1part}) and \(\overline S\bm{\mu}_1=\bm{0}\) and
\begin{equation*}
\overline S\mathcal S_d=
\begin{bmatrix}
0 & I_{n-r+1}
\end{bmatrix},
\end{equation*}
where the zero block has \(r-1\) columns. Consequently,
\(\overline S\mathcal S_d\mathcal S_d^\top\overline S^\top=I_{n-r+1}\). It follows that
\begin{equation*}
\overline{\bm{\mu}}-\bm{\mu}
=
-\nu H_{\nu,d}\overline S^\top\overline S\mathcal S_d\bm{\gamma}.
\end{equation*}
Using \(\mathbb{E}(\bm{\gamma}\bm{\gamma}^\top)=\sigma_\gamma^2 I\), we obtain
\begin{align*}
\mathbb{E}\left\|
\overline{\bm{\mu}}-\bm{\mu}
\right\|_2^2
&=
\sigma_\gamma^2\nu^2
\operatorname{tr}\!\left(
H_{\nu,d}
\overline S^\top\overline S
\mathcal S_d\mathcal S_d^\top
\overline S^\top\overline S
H_{\nu,d}
\right) \\
&=
\sigma_\gamma^2\nu^2
\operatorname{tr}\!\left(
H_{\nu,d}^2\overline U_d
\right) \\
&=
\sigma_\gamma^2\nu
\operatorname{tr}\!\left(
H_{\nu,d}^2\nu\overline U_d
\right).
\end{align*}

Let \(q_\nu(x)=\nu x(1+\nu x)^{-2}\). Since $H_{\nu,d}^2\nu\overline U_d=q_\nu(\overline U_d),$
and \(\overline U_d-U_d\) has rank at most \(r-1\), Lemma
\ref{lem_aux_trace_rank} and the uniform bound
\(\|q_\nu\|_\infty\leq 1/4\) imply that
\begin{equation*}
\operatorname{tr}\!\left(
H_{\nu,d}^2\nu\overline U_d
\right)
=
\operatorname{tr}\!\left(
q_\nu(U_d)
\right)
+
\mathrm{O}(1).
\end{equation*}
By Lemma \ref{thm_aux_trace_approx}, we have 
\begin{equation*}
\operatorname{tr}\!\left(
q_\nu(U_d)
\right)
=
\sum_{j=1}^{n}
\frac{\nu s(\sfx_j)^d}
{(1+\nu s(\sfx_j)^d)^{2}}
+
\mathrm{O}(1).
\end{equation*}
Consequently,
\begin{equation*}
\mathbb{E}\left\|
\overline{\bm{\mu}}-\bm{\mu}
\right\|_2^2
=
\sigma_\gamma^2\nu
\sum_{j=1}^{n}
\frac{\nu s(\sfx_j)^d}
{(1+\nu s(\sfx_j)^d)^{2}}
+
\mathrm{O}(\sigma_\gamma^2\nu),
\end{equation*}
which proves \eqref{eq_biaspart}.

\bibliographystyle{abbrv}
\bibliography{ref}

\end{document}